\documentclass[hidelinks]{article}
\usepackage[a4paper, margin = 0.5in, bottom = 1in]{geometry}
\usepackage{soul}
\usepackage{authblk}
\title{Properties of Hesse derivatives of cubic curves}
\author[1]{Sayan Dutta}
\author[2]{Lorenz Halbeisen}
\author[2]{Norbert Hungerb\"uhler}
\affil[1]{Department of Mathematics and Statistics, IISER Kolkata, Mohanpur, West Bengal 741246, India}
\affil[2]{Department of Mathematics, ETH Zentrum, R\"amistrasse~101, 8092 Z\"urich, Switzerland}
\date{}
\usepackage{amsfonts}
\usepackage{subcaption}
\usepackage{xcolor}
\usepackage{csquotes}
\usepackage{mathtools}
\usepackage{pst-node}
\usepackage{auto-pst-pdf}
\usepackage{float}
\usepackage{amsmath,amssymb,amsthm,comment}
\usepackage{pstricks,pstricks-add,pst-eucl,pst-func,pst-plot,pst-platon}
\usepackage[colorlinks,citecolor=red,urlcolor=blue,linkcolor=blue]{hyperref}
\usepackage{mathrsfs,dsfont}
\usepackage{graphicx}
\usepackage{enumerate}
\usepackage{nicefrac}
\usepackage[font={normalsize},width=1.12\linewidth]{caption}
\usepackage{comment}

\usepackage{pgfplots}
\pgfplotsset{compat=1.15}
\usepackage{mathrsfs}
\usetikzlibrary{arrows}

\newtheorem{nummer}{ }
\newtheorem{thm}[nummer]{\bf Theorem}
\newtheorem{prp}[nummer]{\bf Proposition}
\newtheorem{lem}[nummer]{\bf Lemma}
\newtheorem{cor}[nummer]{\bf Corollary}

\newcommand{\ie} {i.e.}
\newcommand{\eg} {e.g.}

\renewcommand{\phi}{\varphi}
\renewcommand{\theta}{\vartheta}

\newcommand{\G}[1]{\Gamma_{\hspace*{-2pt}{#1}}}

\newcommand{\N}{\mathds{N}}

\newcommand{\R}{\mathds{R}}
\newcommand{\C}{\mathds{C}}

\newcommand{\rpp}{\mathds{R}\mathds{P}^2}
\newcommand{\h}{\raisebox{-.3pt}{\includegraphics[scale=.88]{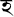}}}
\newcommand{\footnoteh}{\raisebox{-.9pt}{\includegraphics[scale=.75]{haw.eps}}}
\newcommand{\abstracth}{\raisebox{-.1pt}{\includegraphics[scale=.8]{haw.eps}}}
\newcommand{\blueh}{\raisebox{-.1pt}{\includegraphics[scale=1.03]{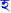}}}

\newcommand{\sectionh}{\raisebox{-.1pt}{\includegraphics[scale=1.26]{haw.eps}}}

\newcommand{\ch}{\raisebox{-0pt}{\includegraphics[scale=1]{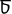}}}
\newcommand{\footnotech}{\raisebox{-.0pt}{\includegraphics[scale=.75]{chaw.eps}}}

\newcommand{\blcb}{\big{\{}}
\newcommand{\brcb}{\big{\}}}

\newcommand {\mult}{\mathbin{*}}

\makeatletter
\def\opargproof[#1]{\par\noindent {\bf #1 }}
 \makeatother

\begin{document}
\begin{center}
{\LARGE\bf Properties of Hesse derivatives of cubic curves}

\bigskip
{\small Sayan Dutta}\\[1.2ex] 
{\scriptsize Department of Mathematics and Statistics, 
IISER Kolkata, Mohanpur, West Bengal 741246, India\\ sd19ms148@iiserkol.ac.in}
\\[1.8ex]

{\small Lorenz Halbeisen}\\[1.2ex] 
{\scriptsize Department of Mathematics, ETH Zentrum,
R\"amistrasse\;101, 8092 Z\"urich, Switzerland\\ lorenz.halbeisen@math.ethz.ch}\\[1.8ex]

{\small Norbert Hungerb\"uhler}\\[1.2ex] 
{\scriptsize Department of Mathematics, ETH Zentrum,
R\"amistrasse\;101, 8092 Z\"urich, Switzerland\\ norbert.hungerbuehler@math.ethz.ch}\\[1.8ex]

\end{center}

\hspace{5ex}{\small{\it key-words\/}: Elliptic curves, Hessian curves, 
geometry of cubic curves, halving formulae, discrete dynamical system}

\hspace{5ex}{\small{\it 2020 Mathematics Subject 
Classification\/}: {\bf 11G05}\,\ 37N99}

\begin{abstract}
The Hesse curve or Hesse derivative $\abstracth \G{f}$ of a cubic curve $\G{f}$ given by a homogeneous polynomial $f$ is the set of  points $P$ such that 
$\det \left (H_f (P)\right )=0$, where $H_f (P)$ is the Hesse matrix of $f$ evaluated at $P$.  Also $\abstracth \G{f}$ is again a cubic curve.
We show that for a point $P\in \abstracth \G{f}$, all the contact points of tangents from $P$ to the curves $\G{f}$ and $\abstracth \G{f}$ 
are intersection points of two straight lines $\ell_1^P$ and $\ell_2^P$ 
(meeting on $\abstracth \G{f}$) with $\G{f}$ and $\abstracth \G{f}$, 
where the product of $\ell_1^P$ and $\ell_2^P$ is the polar conic of $\G{f}$ at $P$. 
The operator $\abstracth$ defines an iterative discrete dynamical system on the set of the cubic curves. We identify the two fixed 
points of this system, investigate orbits that end in the fixed points, and discuss the closed orbits of the dynamical system.
%
\end{abstract}

\section{Introduction}
We will work with cubic curves in the real projective plane $\rpp$. Points  $X=(x_1,x_2,x_3)^T\in\R^3 \setminus \{0\}$ will be denoted by capital letters, the components with small letters, and the equivalence class by $[X]:=\{\lambda X\mid\lambda\in\R\setminus\{0\}\}$.
However, since we mostly work with representatives, we often omit the square brackets in the notation.

Let $f$ be a homogeneous polynomial in the variables $x_1,x_2,x_3$ of degree $3$. Then $f$ defines the projective cubic curve 
$$\G f:= \blcb [X] \in \rpp\mid f(X)=0\brcb\,.$$
The Hesse matrix of $f$ 
is the symmetric $3\times 3$ matrix $\displaystyle{H_f=\Bigl( \frac{\partial^2f}{\partial x_i\partial x_j}\Bigr)}$. 

Observe that $\det(H_f)$ is again a homogeneous cubic polynomial. 
Therefore, we can define the \textit{Hesse derivative} of $\G f$,
denoted $\h\G f$\footnote{In order to denote the Hesse derivative 
of a cubic curve, we
introduce the Bengali letter \footnoteh{}\hspace*{-1pt} (pronounced ``Haw"). 
As a fact we would like to mention that ``Hesse" in Bengali means ``to laugh"!},  
as the cubic curve 
$$\h\G f:=\G{\det(H_f)}= \blcb [X] \in \rpp\mid \det(H_f(X))=0\brcb\,.$$
The polar conic of $\G f$ with respect to the pole $P$ is given by the equation
\begin{equation}
    \mathcal{C}_f (P) :\ \langle X,H_f(P)X\rangle=0\label{eq1}
\end{equation}
or equivalently 
\begin{equation}
    \mathcal{C}_f (P) :\ \langle\nabla f(X), P\rangle = 0.\label{eq2}
\end{equation}
The equivalence of (\ref{eq1}) and (\ref{eq2}) is shown in~\cite{hhs}.
It is clear from~(\ref{eq2}) that the contact points of the tangents from $P$ to $\G f$ are precisely the intersection points of $\mathcal{C}_f(P)$ with $\G f$ (see Figure~\ref{fig1}).

If there is no danger of confusion, we will omit the index and briefly write $\Gamma$ instead of $\G f$. Moreover, we will use the notation $H_\Gamma$ instead of $H_f$, and $\mathcal{C}_\Gamma (P)$ instead of $\mathcal{C}_f (P)$ if the polynomial $f$ is determined by the context or if a general but unique polynomial is meant. 
We would like to mention that the Hesse derivative $\h \Gamma$ 
is also known as \textit{Hessian curve}, denoted $\operatorname{Hess}(\Gamma)$ 
(see, {\eg}, \cite[\S\,4.12,\,p.\,111]{holme}.
However, we prefer the notation $\h\Gamma$ because we want to interpret $\h$ as an operator whose iterations we want to study.
Whenever convenient, we will use $x,y,z$ instead of $x_1,x_2,x_3$ for the coordinates. The figures below of the various projective curves show images of the  curves in the affine plane $x_3=1$ embedded in $\rpp$.

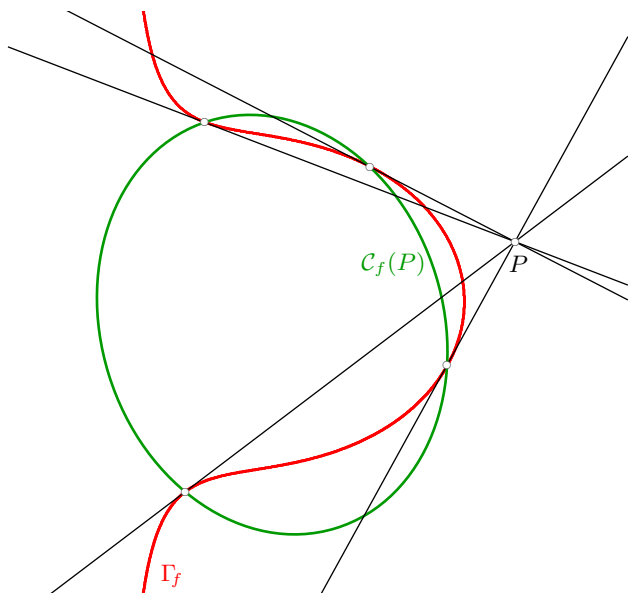
\begin{figure}[H]
  \begin{center}
\newrgbcolor{darkgreen}{0 .6 0}
\psset{xunit=1.4cm,yunit=1.4cm,algebraic=true,dimen=middle,dotstyle=o,dotsize=5pt 0,linewidth=1.6pt,arrowsize=3pt 2,arrowinset=0.25}
\begin{pspicture*}(-1.0857414386694395,-2.7510220414825723)(4.738163172865769,2.754636416416814)
\psplotImp[linewidth=1pt,linecolor=red,stepFactor=0.2](-3.0,-4.0)(5.0,3.0){-4.0+3.0*x^1*y^2-6.0*x^2+2.0*x^3}
\rput{-72.914358}(1.39649387,-0.2162127){\psellipse[linewidth=1pt,linecolor=darkgreen](0,0)(2.018572169301821,1.6064737553715338)}
\psplot[linewidth=0.5pt]{-1.0857414386694395}{4.738163172865769}{(-5.816054349046266--1.130067005618543*x)/-2.916346595806849}
\psplot[linewidth=0.5pt]{-1.0857414386694395}{4.738163172865769}{(--6.946482812425451-2.3687213641519627*x)/-3.098496891409968}
\psplot[linewidth=0.5pt]{-1.0857414386694395}{4.738163172865769}{(-3.3904663678591893--0.7104980909477601*x)/-1.3659904646693124}
\psplot[linewidth=0.5pt]{-1.0857414386694395}{4.738163172865769}{(--3.9353579078398297-1.1689837959452332*x)/-0.638269122223921}
\begin{small}
\psdots[dotsize=3pt 0](3.6773426607464237,0.569346161450969)
\rput[bl](3.62,0.28){\textcolor{black}{$P$}}
\rput[bl](.35,-2.7){\textcolor{red}{$\G f$}}
\rput[bl](2.234626764264113,0.24049180269397283){\textcolor{darkgreen}{$\mathcal{C}_f(P)$}}
\psdots[dotsize=3pt 0](0.5788457693364555,-1.7993752027009935)
\psdots[dotsize=3pt 0](0.7609960649395747,1.6994131670695118)
\psdots[dotsize=3pt 0](2.3113521960771113,1.279844252398729)
\psdots[dotsize=3pt 0](3.0390735385225027,-0.5996376344942643)
\end{small}
\end{pspicture*}
  \caption{A cubic curve $\G f$ and its polar conic $\mathcal{C}_f(P)$ with respect to the pole $P$.
}\label{fig1}
  \end{center}
\end{figure}

It is well known that the polar conic is the product of two projective lines, $\ell_1^P$ and $\ell_2^P$ iff the determinant of the Hesse matrix evaluated at $P$ is equal to 
$0$, {\ie},
$$\mathcal{C}_\Gamma (P)=\langle X, \ell_1^P\rangle \langle X, \ell_2^P\rangle \iff \det \left( H_\Gamma (P)\right)=0.$$
In particular, we obtain the following result
\begin{prp}\label{prop1}
    If $P\in \h\Gamma$, {\ie}, the polar conic is the product of the two lines $\ell_1^P$ and $\ell_2^P$, then the tangents from $P$ to $\Gamma$ touch $\Gamma$ precisely at the points $\Gamma\cap \ell_1^P$ and $\Gamma\cap \ell_2^P$.
\end{prp}

\section{Halving formulae for points on $\sectionh\Gamma$}\label{sec2}
We would now like to compute the contact points of the tangents from a point $P\in \h\Gamma$ to the curve $\h\Gamma$. By a suitable projective transformation, we may assume that the curve is of the form $E_{a,b}$ defined by
\begin{equation}\label{Eab}
E_{a,b}:\ y^2 = x^3 + a\,x^2 + b\,x
\end{equation}
in the affine plane, where\/ $a,\,b\in\R$.
\begin{prp}\label{thm:main}
Let \/ $E_{a,b}$ 
be a non-singular elliptic curve over\/ $\C$ defined by
$$E_{a,b}:\ y^2 = x^3 + a\,x^2 + b\,x$$
where\/ $a,\,b\in\C$, 
and let\/ $P=(x_0,y_0)$ be a point on~$E_{a,b}$.

Let
$$e_1=\frac{-a + \sqrt{a^2 - 4 b}}2,\qquad
e_2 =\frac{-a - \sqrt{a^2 - 4 b}}2.
$$
and let 
$$\gamma =\sqrt{x_0},\qquad
\alpha =\sqrt{x_0-e_1},\qquad
\beta  =\sqrt{x_0-e_2}.$$
Then, \/ $E_{a,b}$ is of the
form
$$y^2=x_0(x_0-e_1)(x_0-e_2)$$
and the\/ $x$-coordinates of the contact points of the tangent of $P$ with $E_{a,b}$, denoted by\/ 
$Q_1,Q_2,Q_3,Q_4$ are\vspace*{-3mm}
\begin{eqnarray*}
x_{11} & = & (\alpha+\gamma)(\beta+\gamma),\\[.8ex]
x_{12} & = & (\alpha-\gamma)(\beta-\gamma),\\[.8ex]
x_{21} & = & (\alpha+\gamma)(-\beta+\gamma),\\[.8ex]
x_{22} & = & (\alpha-\gamma)(-\beta-\gamma).
\end{eqnarray*}
Notice that for the points $Q_i$, $i=1,2,3,4$, we have $2\ast Q_i=-P$, i.e., $\frac P2 = -Q_i$, where $2*Q_i=Q_i+Q_i$ is the usual elliptic curve operation on $E_{a,b}$ (see, {\eg},~\cite{hh22}).
\end{prp}

\begin{proof} We obviously have $-e_1-e_2=a$ and $e_0e_1+e_1e_2+e_2e_0=b$.

To show that $x_{11},x_{12},x_{21},x_{22}$ are the 
the $x$-coordinates of points $Q\in E_{a,b}$ such that $2\mult Q=P$,
it is enough to show that the $x$-coordinate 
of the point $Q_{ij}:=(x_{ij},y)$, 
where $i,j\in\{1,2\}$ and $y=\sqrt{x_{ij}^3+a\,x_{ij}^2+b\,x_{ij}}$,
is equal to~$x_0$. Now, the $x$-coordinate $x_{2ij}$
of the point $2\mult Q_{ij}$
is given by the formula
$$x_{2ij}=\frac{x_{ij}^4-2b\,x_{ij}^2+b^2}
{4(x_{ij}^3+a\,x_{ij}^2+b\,x_{ij})}=\frac{(x_{ij}^2-b)^2}
{4x_{ij}(x_{ij}^2+a\,x_{ij}+b)}.$$
Furthermore, we have $a=\alpha^2 + \beta^2 - 2\gamma^2$ 
and $b=(\alpha^2 - \gamma^2) (\beta^2 - \gamma^2)$, and if we 
write $x_{ij},a,b$ in terms of $\gamma,\alpha,\beta$, it is
not hard to verify that 
$$(x_{ij}^2-b)^2=4x_{ij}\gamma^2(x_{ij}^2+a\,x_{ij}+b),$$
which shows that $x_{2ij}=x_0$.
\end{proof}
\section{Intersection of $\mathcal C_\Gamma(P)$ with $\sectionh \Gamma$ for $P\in \sectionh \Gamma$}
In this section, we combine the results from Section~\ref{sec2} with the property that for every point $P\in \h \Gamma$, the polar conic of $\Gamma$ with respect to the pole $P$ is the product of two lines, $\ell_1^P$ and $\ell_2^P$ (see Proposition~\ref{prop1}). In particular, we want to show that the two lines $\ell_1^P$ and $\ell_2^P$ intersect on the curve $\h \Gamma$ and that the points $x_{ij}$, $i,j=1,2$ correspond to the other intersection points of $\ell_1^P$ and $\ell_2^P$ with $\h \Gamma$ (see Figure~\ref{fig:galaxy1}).

For this, we start with the cubic curve
$$\G {a,b} :\ ax^3 + 3xy^2 +3bx^2z - b^2z^3 = 0$$
with $a,b\in \R$, $b\neq 0$.
For $\G {a,b}$, we get
$$H_{\G {a,b}}:=\begin{pmatrix}
                    6ax+6bz& 6y& 6bx\\
                    6y& 6x& 0\\
                    6bx& 0& -6b^2z
                \end{pmatrix}$$
and hence
$$\h\G {a,b} :\ y^2z=x^3+ax^2z+bxz^2.$$
In other words, $\h\G {a,b}=E_{a,b}$, as introduced in~(\ref{Eab}) in the previous section.

By definition, if $P=(x_0,y_0)\in \h \G {a,b}$, then $\det \left (H_{\G {a,b}}(P)\right) = 0$, which implies that the conic section
$$\mathcal C_{\G {a,b}} :\ (ax_0+b)x^2 + x_0y^2 + 2y_0xy + 2bx_0x - b^2$$
can be written as the product of two lines $\ell_1^P$ and $\ell_2^P$. In the following lemma, we will compute these two lines in terms of $e_1$ and $e_2$ defined in the previous section.

\begin{lem}
    The lines $\ell_1^P$ and $\ell_2^P$ are given by
    \begin{align*}
        \ell_1^P &:\ ux+vy+wz=0\\
        \ell_2^P &:\ rx+sy+tz=0
    \end{align*}
    where\vspace*{-3mm}
    \begin{align*}
        u&=-x_0-\sqrt{(e_1-x_0)(e_2-x_0)}\\
        v&=-\sqrt{x_0}\\
        w&=e_1e_2\\
        r&=-x_0+\sqrt{(e_1-x_0)(e_2-x_0)}\\
        s&=\sqrt{x_0}\\
        t&=e_1e_2.
    \end{align*}
\end{lem}
\begin{proof}
    We first replace $y_0$ by $\sqrt{x_0^3+ax_0^2+bx_0}$, 
    and $a$, $b$ by $-e_1-e_2$ and $e_1e_2$ respectively. Then, we have
    $$\mathcal C_{\G {a,b}} :\ (e_1e_2-e_1x_0-e_2x_0)\;x^2 + x_0\;y^2 + 2\sqrt{x_0^3-e_1x_0^2-e_2x_0^2+e_1e_2x_0}\;xy + 2e_1e_2x_0\;x - e_1^2e_2^2.$$
    This leaves us with the easy exercise to check that
    \begin{align*}
        &ur=-e_1e_2+e_1x_0+e_2x_0\\
        &us+vr=-2\sqrt{x_0^3-e_1x_0^2-e_2x_0^2+e_1e_2x_0}\\
        &vs=-x_0^2\\
        &ut+wr=-2e_1e_2x_0\\
        &wt=-e_1^2e_2^2\\
        &vt+ws=0
    \end{align*}
    hence completing the proof.
\end{proof}

Before we compute the intersection points of $\ell_1^P$ and $\ell_2^P$ with the curve $E_{a,b}$, we show that the two lines intersect on the curve $E_{a,b}$.

\begin{lem}
     Let $P=(x_0,y_0)$ be a point on $\h\G {a,b}$. Then the point $\ell_1^P \cap \ell_2^P =:S=(x_S,y_S)$ lies on the same curve $\h\G {a,b}=E_{a,b}$.
\end{lem}
\begin{proof}
    Let us rewrite the lines as
    $$
        \ell_1^P :\ y=\frac{-u}{v}x+\frac{-w}{v}\qquad\text{and}\qquad
        \ell_2^P :\ y=\frac{-r}{s}x+\frac{-t}{s}
    $$
    which shows that the intersection of the two lines is given as:
    \begin{align*}
        x_S&=\frac{tv-ws}{us-rv}=\frac {e_1e_2}{x_0}=\frac b{x_0}\\
        y_S&=\frac{-u}{v}x+\frac{-w}{v}=-\frac{by_0}{x_0^2}
    \end{align*}
    Now, we use the fact that
    $$y_0^2=x_0^3+ax_0^2+bx_0$$
    to show that
    $$\Bigl(\frac{b}{x_0}\Bigr)^3+a\Bigl(\frac{b}{x_0}\Bigr)^2+b\Bigl(\frac{b}{x_0}\Bigr)=\Bigl(\frac{-by_0}{x_0^2}\Bigr)^2$$
    hence proving that $S\in E_{a,b}$.
\end{proof}

\begin{lem}
    The map   $\h \G {a,b}\to\h \G {a,b},\ P=(x_0,y_0)\mapsto S=(x_S,y_S)$, is an involution. 
\end{lem}
\begin{proof}
    We check that
    $$
    \frac{b}{\frac{b}{x_0}}=x_0\qquad\text{and}\qquad
    \frac{-by_0}{x_0^2}=\frac{-b\cdot\frac{-by_0}{x_0^2}}{\frac{b^2}{x_0^2}}=y_0
    $$
    hence completing the proof.
\end{proof}
Now, we show that the other intersection points of $\ell_1^P$ and $\ell_2^P$ with $E_{a,b}$ are exactly the points $x_{ij}$, $i,j=1,2$ from Proposition~\ref{thm:main}.

\begin{lem}
    Besides the point $S$, the intersection points of $\ell_1^P$ and $\ell_2^P$ with $E_{a,b}$ are exactly the points $Q_i$, $i=1,2,3,4$, with $2\ast Q_i=-P$. More precisely, the points $x_{11}$ and $x_{12}$ are on the line $\ell_1^P$ and the points $x_{21}$ and $x_{22}$ are on $\ell_2^P$.
\end{lem}
\begin{proof}
To find the $x$-coordinate of  the intersection points of $\ell_i^P$ with $E_{a,b}$ we eliminate $y$ form the equations for $E_{a,b}$ and $\ell_1^P$, and $\ell_2^P$, respectively.
    The resulting equations are of degree $3$ in $x$, but since we already know the root $\frac b{x_0}$, the problem reduces to quadratic equations
    $$x^2-2\left(x_0\pm \alpha\beta\right)x+e_1e_2=0.$$
    The solutions are $$x_0+\alpha\beta\pm \sqrt{\left(x_0+\alpha\beta\right)^2-e_1e_2}\qquad\text{and}\qquad
    x_0-\alpha\beta\pm \sqrt{\left(x_0-\alpha\beta\right)^2-e_1e_2}$$
    and one checks easily that these expressions agree with the formulas for $x_{ij}$ from Proposition~\ref{thm:main}.
    
    It remains to show that the $y$-coordinates match as well. To see that, let us denote by $A$ and $B$, the points at which the tangents from $P$ to $E_{a,b}$ meet $E_{a,b}$. So, for our claim to be true, we have
    \begin{align*}
        &A+B=-S\\
        &2\ast A=-P\\
        &2\ast B=-P\\
    \end{align*}
    which implies
    $$2\ast P= 2\ast S$$
    and hence it is enough to show that this is indeed true. 

    To do so, we note that a formula for doubling the point $P=(x_0,y_0)$ on $E_{a,b}$ is given by
    $$2\ast P=\left(\frac{\left(x_0^2-e_1e_2\right)^2}{4y_0^2}, \frac{\left(x_0^2-e_1e_2\right)\left(e_1e_2-2e_1x_0+x_0^2\right)\left(e_1e_2-2e_2x_0+x_0^2\right)}{8y_0^3}\right)$$
    and hence
    $$2\ast S=\left(\frac{\left(\frac{e_1^2e_2^2}{x_0^2}-e_1e_2\right)^2}{4\cdot\frac{e_1^2e_2^2y_0^2}{x_0^4}}, \frac{\left(\frac{e_1^2e_2^2}{x_0^2}-e_1e_2\right)\left(e_1e_2-2e_1\cdot\frac{e_1e_2}{x_0}+\frac{e_1^2e_2^2}{x_0^2}\right)\left(e_1e_2-2e_2\cdot\frac{e_1e_2}{x_0}+\frac{e_1^2e_2^2}{x_0^2}\right)}{-8\cdot\frac{e_1^3e_2^3y_0^3}{x_0^6}}\right)$$
    since $S=\left(\frac{e_1e_2}{x_0},\frac{-e_1e_2y_0}{x_0^2}\right)$.

    The equality of $2\ast P$ and $2\ast S$ immediately follows by multiplying the numerator and denominator $x$ coordinate of $2\ast S$ by $\frac{x_0^4}{e_1^2e_2^2}$ and the $y$ coordinate by $\frac{-x_0^6}{e_1^3e_2^3}$.
\end{proof}

\begin{thm}\label{thm:7}
    Let $\Gamma$ be a cubic curve and let $P\in \h \Gamma$. Then, all the contact points of tangents from $P$ to the curves $\Gamma$ and $\h \Gamma$ are intersection points of $\ell_1^P$ and $\ell_2^P$ with $\Gamma$ and $\h \Gamma$. In addition, the intersection $Q$ of $\ell_1^P$ and  $\ell_2^P$ lies on $\h \Gamma$ (see  Figure~\ref{fig:galaxy1}).
\end{thm}

\begin{figure}[H]
    \begin{center}
    \newrgbcolor{darkgreen}{0 .6 0}
    \psset{xunit=.5cm,yunit=.5cm,algebraic=true,dimen=middle,dotstyle=o,dotsize=5pt 0,linewidth=1.6pt,arrowsize=3pt 2,arrowinset=0.25}
\begin{pspicture*}(-12,-8)(12,9)
\psplotImp[linewidth=1pt,linecolor=red,stepFactor=0.2](-21.0,-12.0)(14.0,11.0){-27.0+73.76537180435969*y^2-31.176914536239792*x^1*y^2+73.76537180435969*x^2+10.392304845413264*x^3}
\psplotImp[linewidth=1pt,linecolor=blue,stepFactor=0.2](-21.0,-12.0)(14.0,11.0){27.0+19.76537180435968*y^2+31.176914536239792*x^1*y^2+19.76537180435968*x^2-10.392304845413264*x^3}

\psplot[linewidth=1pt,linecolor=darkgreen]{-19.99853147004773}{13.932821077832964}{(-39.69239667605544--1.1543667563422733*x)/15.822658235352947}
\psplot[linewidth=1pt,linecolor=darkgreen]{-19.99853147004773}{13.932821077832964}{(-2.1291893393761985-8.119971964761655*x)/-9.135955598608435}


\psplot[linewidth=0.5pt,linecolor=red]{-19.99853147004773}{13.932821077832964}{(--34.67247972264892--5.723645661265573*x)/10.659785235528956}
\psplot[linewidth=0.5pt,linecolor=red]{-19.99853147004773}{13.932821077832964}{(-29.846836876829634-5.430852406960389*x)/-8.686297570583545}
\psplot[linewidth=0.5pt,linecolor=red]{-19.99853147004773}{13.932821077832964}{(-1.722894256112529-2.396326303496082*x)/1.523829636920521}
\psplot[linewidth=0.5pt,linecolor=blue]{-19.99853147004773}{13.932821077832964}{(--6.364114857914801-4.276485650618115*x)/7.136360664769403}
\psplot[linewidth=0.5pt,linecolor=red]{-19.99853147004773}{13.932821077832964}{(--1.8375274275844449-1.6019801794243502*x)/2.4175655945104486}
\psplot[linewidth=0.5pt,linecolor=blue]{-19.99853147004773}{13.932821077832964}{(-7.20541276356701-4.709067807039191*x)/1.207049567376623}
\begin{small}
\psdots[dotsize=3pt 0,linecolor=blue](-2.0778687703312855,2.1369562938978546)
\rput[bl](-2.03,2.438624270316532){\blue{$P$}}
\psdots[dotsize=3pt 0,linecolor=black](-10.76416634091483,-3.293896113062534)
\psdots[dotsize=3pt 0,linecolor=black](-0.5540391334107644,-0.2593700095982276)
\psdots[dotsize=3pt 0,linecolor=black](0.339696824179163,0.5349761144735043)
\psdots[dotsize=3pt 0,linecolor=black](8.58191646519767,7.860601955163428)
\psdots[dotsize=3pt 0,linecolor=black](-0.8708192029546624,-2.5721115131413366)
\psdots[dotsize=3pt 0,linecolor=black](5.058491894438117,-2.139529356720261)
\psdots[dotsize=3pt 0,linecolor=black](-3.3605217194792876,-2.7537516506319566)
\rput[bl](-3.6093587588140945,-3.6){\blue{$Q$}}
\rput[bl](-7,-2.83){\textcolor{darkgreen}{$\ell_1^P$}}
\rput[bl](3.18,2.5){\textcolor{darkgreen}{$\ell_2^P$}}
\rput[bl](2.9,8){\red{$\Gamma$}}
\rput[bl](6,2){\blue{$\blueh\hspace*{-1pt}\Gamma$}}
\end{small}
\end{pspicture*}
    \caption{Illustration for Theorem~\ref{thm:7} for a cubic curve $\Gamma$ with the symmetry group of an equilateral triangle (see~\cite{hhh}).  The Hesse derivative $\h\Gamma$ has the same symmetry.  The curve $\Gamma$ is given by $2 \sqrt{3} x^3+9 (\sqrt{3}+1) (x^2+y^2)z-6 \sqrt{3} x y^2-9 z^3=0$ and has the property that $\h^2\Gamma=\Gamma$ 
    (see Section~\ref{sec:loops}).}
    \label{fig:galaxy1}
    \end{center}
\end{figure}
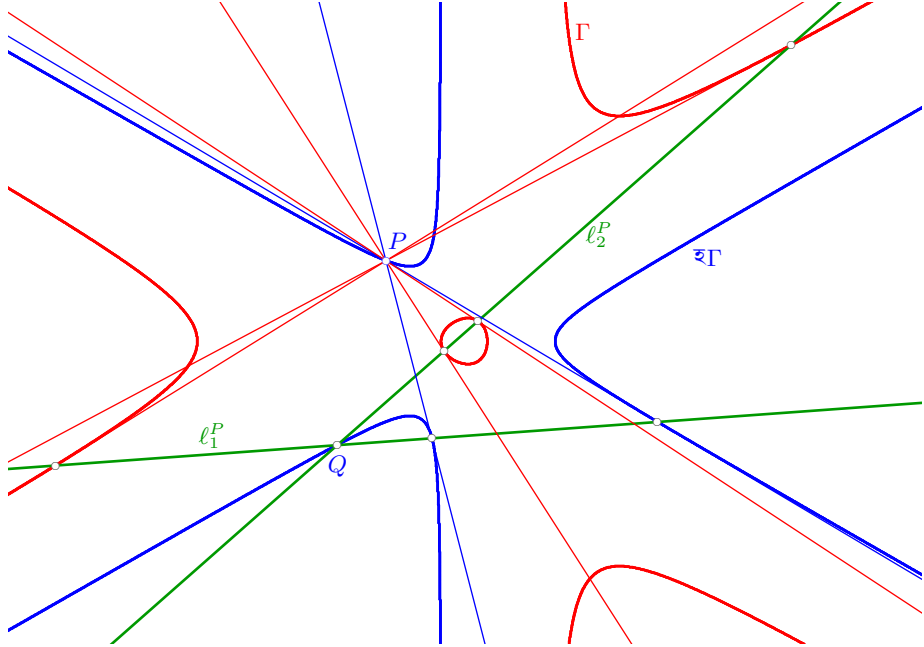

\section{Hesse Form of Cubic Curves}
In this section, we consider a cubic curve in its Hesse form
$$\G c :\ x^3+y^3+z^3+c\,xyz=0$$
with $c\in\R$. Notice $\G{-3}$ is a degenerate curve. Formally, we put
$$
\G{\infty} :\ xyz=0.
$$
\begin{lem}\label{c0}
Let $c_0\neq 0$. Then the Hesse derivative of 
$\G{c_0}$ is $\h \G{c_0}=\G{c_1}$ where
$$c_1=-\frac{108+c_0^3}{3c_0^2}.$$
The Hesse derivative of $\G{0}$ is $\h\G0=\G{\infty}$, and
the Hesse derivative of $\G{\infty}$ is $\h\G{\infty}=\G{\infty}$.
\end{lem}
\begin{proof}
   We have
    $$H_{\G{c_0}}(x,y,z):=\begin{pmatrix}
                    6x& c_0z& c_0y\\
                    c_0z& 6y& c_0x\\
                    c_0y& c_0x& 6z
                \end{pmatrix}.$$
    This yields $\det H_{\G{c_0}} (x,y,z)= 
    - 6 c^2 (x^3 + y^3 + z^3) + 2 (108 + c^3) x y z $, and the claim follows for $c_0\in\R\setminus\{0\}$. The cases $\G0$ and 
    $\G{\infty}$ are also easily checked.
\end{proof}
An immediate corollary is
\begin{cor}
    Let $c_0\neq 0$.
    Then, the $(n+1)$-th Hesse derivative of $\G{c_0}$ is given by
    $$\h^{n+1}\G{c_0} :\ x^3+y^3+z^3+c_{n+1} xyz=0$$
    where $c_{n+1}=-\frac{108+c_n^3}{3c_n^2}$ for every $n\ge 0$, as long as $c_n\neq 0$.
\end{cor}

\section{Analysis of iterates}
Motivated by Lemma~\ref{c0}, we consider the function
\begin{equation}h(x)=\frac{a+x^3}{bx^2}\label{h}\end{equation}
for $a\in\R$ and $b\in\R\setminus\{0\}$.

\begin{lem}\label{lem:10}
    The function $h$ defined in~(\ref{h}) has a pole at $x=0$, and an oblique asymptote $y=\frac x b$. For $b\neq -1$
    $$\phi:=\sqrt[3]{\frac a{b-1}}$$
    is the unique real fixed point of $h$. The function $h$ has the unique critical point $\kappa :=\sqrt[3]{2a}$ with critical value
    $$h(\kappa)=\frac{3a^{1/3}}{2^{2/3}b}.$$
\end{lem}
The proof is elementary.

\noindent{\it Remark:} In our case, $b=-3$, we have 
$$h(\phi)=\phi=-\sqrt[3]{\frac{a}{4}}=-\sqrt[3]{\frac{a^3}{4a^2}}=h(\kappa).$$
This case also gives us the crucial property for the partition of $\R\setminus \{\phi\}$ into the intervals $N=(-\infty, \phi)$ and $P=(\phi,\infty)$. Namely, we have
$x\in P\setminus\{0\}$ iff $h(x)\in N$, and 
$x\in N$ iff $h(x)\in P$. For the next two propositions, we will assume $b=-3$ and $a\neq 0$, and hence $\phi\neq 0$.
\begin{prp}Let $b=-3$ and $a\neq 0$.
    If we define
    $$h^{(n)}:=\underbrace{h\circ h\circ \dots h}_{n \text{ times}}$$
    then, $y=\frac x{b^n}$ is an oblique asymptote of $h^{(n)}$.
    Furthermore, if $\kappa_n$ is a critical point of $h^{(n)}$, we have
    $$h^{(n)}(\kappa_n)=\phi =-\sqrt[3]{\frac{a}{4}}.$$
    Conversely, if $h^{(n)}(x)=\phi$, then either $x=\phi$ or $\frac{d}{dx}h^{(n)}(x)=0$.
\end{prp}
\begin{proof}
    Since we already know the oblique asymptote of $h$ from Lemma~\ref{lem:10}, we can now inductively argue that
    \begin{align*}
        \lim_{x\to \pm\infty}\left|h^{(n+1)}(x)-\frac x{b^{n+1}}\right|&=\lim_{x\to \pm\infty}\left | \frac{a+\left(h^{(n)}(x)\right)^3}{b\left(h^{(n)}(x)\right)^2} - \frac{\frac{x}{b^n}}{b}\right|\\
        &=\lim_{x\to \pm\infty}\left| \frac a{b\left(h^{(n)}(x)\right)^2}+\frac{h^{(n)}(x)-\frac x{b^n}}{b}\right|=0
    \end{align*}
    using the fact that $\left(h^{(n)}(x)\right)^2\to\infty$ for $x\to \pm\infty$. Note that in this part did we did not need the assumption $b=-3$.
    
    For the next part, observe first that we have
 $$h(x)=\varphi \iff (x-\varphi) (x-\kappa)^2=0.$$
This equation has only two solutions, namely $x_1=\varphi$ and $x_2=\kappa$.
    Now, by chain rule, we obtain
    \begin{align}
        &\frac{d}{dx}h^{(n)}(x)=\prod_{r=0}^{n-1}h^\prime\left(h^{(r)}(x)\right)=0\label{ableitung}\\
        \iff &h^\prime\left(h^{(r)}(x)\right)=0 \text{ for some } r\in \{0,1,\dots ,n-1\}\notag\\
        \iff &h^{(r)}(x)= \kappa\text{ for some } r\in \{0,1,\dots ,n-1\}.\notag
    \end{align}
    So, it follows immediately that $\frac d{dx}h^{(n)}(x)=0$ implies
    $h^{(n)}(x)=\phi.$

Now we prove the converse, as stated in the lemma.    
For $n=0$, the statement is trivially true. Assume that for some $n\ge 1$ we have that $h^{(n)}(x)=\varphi$ implies that either $x=\varphi$ or $\frac d{dx}h^{(n)}(x)=0$.
Then we have for $h^{(n+1)}(x)=\varphi$ that $h^{(n)}(x)=\kappa$ or $h^{(n)}(x)=\varphi$.  On the other hand,
$\frac d{dx}h^{(n+1)}(x)=h'(h^{(n)}(x)) \frac d{dx}h^{(n)}(x)$. If $h^{(n)}(x)=\kappa$, then the first factor in this product is zero and the derivative of $\frac d{dx}h^{(n+1)}(x)$ vanishes.
If $h^{(n)}(x)=\varphi$, then, by induction, $x=\varphi$ or $\frac d{dx}h^{(n)}(x)=0$, and again the derivative of $\frac d{dx}h^{(n+1)}(x)$ is zero.
\end{proof}
\begin{prp}\label{number_of_critical_points}
    Let $\chi_n$ be the number of critical points of $h^{(n)}$. Then, the sequence $\{\chi_n\}$ is given by
    $$\chi_{2r+1}=2\times 3^r-1\;\;\text{ and }\;\; \chi_{2r}=3^r-1$$
    for all $r\ge 0$. This corresponds to OEIS A062318.
\end{prp}
\begin{proof}
    Without only carry out the case $\phi<0$. The proof for $\phi>0$ is essentially the same.

%
%
%

Let $N=(-\infty,\varphi)$ and $P=(\varphi, \infty)$. 
    Observe first that for given $y\in N$ the equation
    $$y=h(x)\;\;\text{ or equivalently }\;\; x^3+3yx^2-4\phi^3=0$$
    has three distinct real roots. Indeed, the discriminant
    $\Delta = 27\times 16\phi^3\left( y^3-\phi^3\right)$ is strictly positive for $y\in N$.
 Moreover, if $x,y\in N$ then the expression $x^3+3yx^2-4\phi^3$
    is strictly negative, hence the three solutions of $y=h(x)$ must lie in $P$. 
Hence, the preimage $h^{-1}(y)$ of a point in $y\in N$ has cardinality 3, and lies in $P$. 
Similarly, the preimage $h^{-1}(y)$ of a point in $y\in P$ has cardinality 1, and lies in $N$. 

Now, for $n=1$, the set of critical points of $h$ is $C_1=\{\kappa\}\subset P$.
Let $S_1:=h^{-1}(C_1)$, and $S_k:=h^{-1}(S_{k-1})$ for $k>1$.
For $n>1$ can read of from equation~(\ref{ableitung}) that the
the set of critical points of $h^{(n)}$ is the set $C_n=C_{n-1}\cup S_{n-1}$.
Observe that $S_n\subset N$ if $n$ is odd, and $S_n\subset P$ if $n$ is even.
Hence we have $\operatorname{card}C_{2n}=\operatorname{card}C_{2n}+3^{n-1}$
and $\operatorname{card}C_{2n+1}=\operatorname{card}C_{2n}+3^{n}$. This corresponds to the sequence OEIS A062318.
\end{proof}

\begin{prp}\label{number of fixed points}
    Let $\Phi_n$ be the number of fixed points of $h^{(n)}$. Then, the sequence $\{\Phi_n\}$ is given by
    $$\Phi_{2r+1}=1\;\;\text{ and }\;\;\Phi_{2r}=2\chi_{2r}-1=2\times 3^r-3$$
    for all $r\ge 0$.
\end{prp}
\begin{proof}
    Since $h$ maps $N$ to $P$ and vice versa, the only fixed point of $h^{(2r+1)}$ is $\varphi$.

    For the fixed points of $h^{(2r)}$, we begin by assuming without loss of generality that $a>0$, $\phi<0$, and recalling that $h$ has a pole at $x=0$, 
    it is decreasing in the intervals $(-\infty,0)$, $(\kappa,\infty)$ and increasing in $(0,\kappa)$. 
    We also know that the only critical point of $h$ is at $\kappa$ and it has a local maximum there. So, for convenience, we define
    $$\tilde{h}(x):=h(x+\phi)-\phi$$
    {\ie}, we shift the point $(\phi,\phi)$ to the origin. The number of fixed points of $h^{(n)}$ is the
    same as the number of fixed points of $\tilde h^{(n)}$. Consider the sets $\tilde N=(-\infty, 0), \tilde P=(0,\infty)$, and
    $A=(0,\varphi), B=(\varphi,\varphi+\kappa), C=(\kappa,\infty)$. 
Then, $\tilde h$ maps $\tilde N$ bijectively to $\tilde P$, and $A,B$ and $C$ each bijectively to $\tilde N$.
Hence $\tilde h$ maps $\tilde N,\tilde P$ to  $\tilde P$ and tree copies of $\tilde N$. So after $2r$ iterations, the range of $\tilde h^{(2r)}$
consists of $3^r$ copies of $\tilde N$ and $3^r$ copies of $\tilde P$.
Figure~\ref{fig-zigzag} shows schematically the behaviour of $\tilde h^{(2r)}$.
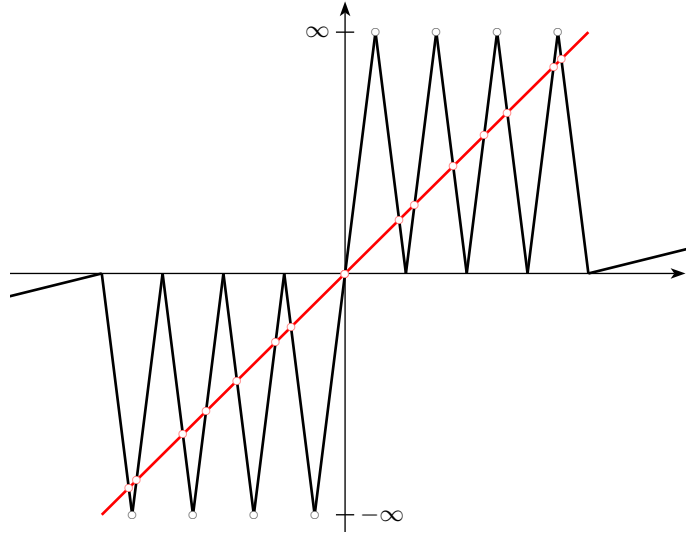
\begin{figure}[ht]
\begin{center}
\psset{xunit=.4cm,yunit=.4cm,algebraic=true,dimen=middle,dotstyle=o,dotsize=5pt 0,linewidth=1.6pt,arrowsize=3pt 2,arrowinset=0.25}
\begin{pspicture*}(-11.00542704053232,-8.55075404872461)(11.195258877733565,9.021217430983894)

\psline[linewidth=.5pt, arrows=->](-12,0)(11.195,0)
\psline[linewidth=.5pt, arrows=->](0,-10)(0,9.02)

\psline[linewidth=1pt](0.,0.)(1.,8.)
\psline[linewidth=1pt](1.,8.)(2.,0.)
\psline[linewidth=1pt](2.,0.)(3.,8.)
\psline[linewidth=1pt](3.,8.)(4.,0.)
\psline[linewidth=1pt](4.,0.)(5.,8.)
\psline[linewidth=1pt](5.,8.)(6.,0.)
\psline[linewidth=1pt](6.,0.)(7.,8.)
\psline[linewidth=1pt](7.,8.)(8.,0.)
\psline[linewidth=1pt](8.,0.)(16.,2.)

\psdots[dotsize=3pt 0](1,8)
\psdots[dotsize=3pt 0](3,8)
\psdots[dotsize=3pt 0](5,8)
\psdots[dotsize=3pt 0](7,8)


\psline[linewidth=1pt](0.,0.)(-1.,-8.)
\psline[linewidth=1pt](-1.,-8.)(-2.,0.)
\psline[linewidth=1pt](-2.,0.)(-3.,-8.)
\psline[linewidth=1pt](-3.,-8.)(-4.,0.)
\psline[linewidth=1pt](-4.,0.)(-5.,-8.)
\psline[linewidth=1pt](-5.,-8.)(-6.,0.)
\psline[linewidth=1pt](-6.,0.)(-7.,-8.)
\psline[linewidth=1pt](-7.,-8.)(-8.,0.)
\psline[linewidth=1pt](-8.,0.)(-16.,-2.)

\psline[linewidth=.5pt](-.3,8)(.3,8)
\rput[r](-.5,8){\small{$\infty$}}

\psline[linewidth=.5pt](-.3,-8)(.3,-8)
\rput[l](.5,-8){\small{$-\infty$}}

\psdots[dotsize=3pt 0](-1,-8)
\psdots[dotsize=3pt 0](-3,-8)
\psdots[dotsize=3pt 0](-5,-8)
\psdots[dotsize=3pt 0](-7,-8)

\psplot[linewidth=1pt,linecolor=red]{-8}{8}{(-0.-1.*x)/-1.}
\begin{scriptsize}
\psdots[dotsize=3pt 0,linecolor=red](7.111111111111111,7.111111111111111)
\psdots[dotsize=3pt 0,linecolor=red](6.857142857142857,6.857142857142857)
\psdots[dotsize=3pt 0,linecolor=red](5.333333333333333,5.333333333333333)
\psdots[dotsize=3pt 0,linecolor=red](4.571428571428571,4.571428571428571)
\psdots[dotsize=3pt 0,linecolor=red](3.5555555555555554,3.5555555555555554)
\psdots[dotsize=3pt 0,linecolor=red](2.2857142857142856,2.2857142857142856)
\psdots[dotsize=3pt 0,linecolor=red](1.7777777777777777,1.7777777777777777)
\psdots[dotsize=3pt 0,linecolor=red](0.,0.)
\psdots[dotsize=3pt 0,linecolor=red](0.,0.)
\psdots[dotsize=3pt 0,linecolor=red](-1.7777777777777777,-1.7777777777777777)
\psdots[dotsize=3pt 0,linecolor=red](-2.2857142857142856,-2.2857142857142856)
\psdots[dotsize=3pt 0,linecolor=red](-3.5555555555555554,-3.5555555555555554)
\psdots[dotsize=3pt 0,linecolor=red](-4.571428571428571,-4.571428571428571)
\psdots[dotsize=3pt 0,linecolor=red](-5.333333333333333,-5.333333333333333)
\psdots[dotsize=3pt 0,linecolor=red](-6.857142857142857,-6.857142857142857)
\psdots[dotsize=3pt 0,linecolor=red](-7.111111111111111,-7.111111111111111)
\end{scriptsize}
\end{pspicture*}
\caption{Schematic profile of the function $\tilde h^{(2r)}$: We have $(3^r-1)/2$ spikes on the positive $x$-axis and $(3^r-1)/2$ on the
negative $x$-axis.}\label{fig-zigzag}
\end{center}
\end{figure}

Observe that the oblique asymptote of $\tilde h^{(2r)}$ is given by $y=x/3^{2r}$, hence the line $y=x$ does not intersect the leftmost and the rightmost branch of the graph of $\tilde h^{(2r)}$.
Hence the number of fixed points of $\tilde h^{(2r)}$ is $4\cdot\frac{3^r-1}2-1=2\times 3^r-3$.
    \end{proof}


Analogous to the determination of the number of fixed points of $h^{(n)}$, the number of zeros is now calculated.
\begin{prp}\label{number of zeros}
    Let $\rho_n$ be the number of zeros of $h^{(n)}$. Then, the sequence $\{\rho_n\}$ is given by
    $$\rho_{2r}=\rho_{2r+1}=3^r$$
    for all $r\ge 0$.
\end{prp}
\begin{proof}
The number of zeros of $h^{(n)}$ equals the number of solutions of the equation $\tilde h^{(n)}=-\varphi$.
We continue to assume that $\varphi$ is negative. If $n=2r$ is even, then each spike on the positive $x$-axis
contributes  two solutions, and the rightmost branch another one. So, the number of solutions
of $\tilde h^{(2r)}=-\varphi$ is $2\cdot\frac{3^r-1}2+1=3^r$, as claimed.

Similarly, for $n=2r+1$, the function $\tilde h^{(2r+1)}$ exhibits $(3^r-1)/2$ spikes in the second quadrant, is negative 
on the positive $x$-axis, and approaches the asymptote $-x/3^{2r+1}$. Therefore, we also find $2\cdot\frac{3^r-1}2+1=3^r$ 
zeros of $\tilde h^{(2r+1)}$.
%
%
%
%
%
\end{proof}
\section{Loops and chains of Hesse derivatives}\label{sec:loops}
We return to considering the curves in Hesse form, {\ie},
\begin{equation}\G c :\ x^3+y^3+z^3+c\,xyz=0, \qquad \G\infty: \ xyz=0\label{hess}\end{equation}
for $c\in \R$. So, in this section, while referring to the function $h$ from the previous section, we will assume $a=108$, $b=-3$, $\phi=-3$, $\kappa=6$
(see Lemma~\ref{c0} and Lemma~\ref{lem:10}).

We will begin with the following observation.
\begin{lem}\label{components}
    The geometric interpretation of the property $x\in P\iff h(x)\in N$ is that the curve $\G{c}$ has two components when $x<-3$ and only one component when $x>-3$.
    In particular, if $\Gamma$ has one component, then $\h\Gamma$ has two components, and if $\Gamma$ has two component, then $\h\Gamma$ has one components.
\end{lem}
\begin{proof}
    We begin by noting that $\G c$ is unchanged under the transformation $(x,y)\mapsto (y,x)$, and hence it is symmetric with respect to the line $y=x$.
    Now, let us calculate the number of intersection points of $\G{c}$ with the line $y=x$. This gives us the equation
    $$2x^3+cx^2+1=0$$
    the discriminant of which is
    $$-27\times 4-4c^3$$
    and hence the equation has three real roots iff $c<-3$.
\end{proof}
The operator $\h$ defines via $\Gamma\mapsto \h\Gamma$ an iterative discrete dynamical system on the set of the cubic curves~(\ref{hess}) in Hesse form.
The dynamics is given by  Lemma~\ref{c0}. The system has exactly two fixed points, namely $\G{-3}$ and $\G\infty$. We are now interested in
orbits of a given length which end in one of the fixed points, and in closed orbits of a given length. The former we call {\em Hesse chains}, the latter {\em Hesse loops}.
So, a Hesse chain is given by
$$\h^n\G{c_0}=\G{c_n}=\G{-3}\qquad\text{or}\qquad\h^n\G{c_0}=\G{c_n}=\G{\infty}$$
where we call the minimal $n$ with this property the length of the chain.
Similarly, a Hesse loop is given by
$$\h^n\G{c_0}=\G{c_n}=\G{c_0}$$
where the minimal $n>0$ with this property is the length of the loop.

The number of Hesse chains ending in $\G{-3}$ of length $n$ is easy to calculate as shown in the following Lemma.
\begin{lem}
    If $\ch_n^{(-3)}$ denotes the number\,\footnote{We again find ourselves at a loss of expressions to notate the ``ch" sound used in ``chain" as such an alphabet is not present in English, Latin or Greek script. We will use this excuse to use another Bengali alphabet, namely $\footnotech$, pronounced as \textit{``chaw"}.}  of Hesse chains ending in $\G{-3}$ of length $n$, then
    $$\ch_{2r}=\ch_{2r-1}=3^{r-1}$$
    for $r\ge 1$.
\end{lem}
\begin{proof}
    From Proposition~\ref{number_of_critical_points}, it is clear that
    $$\ch_{2r}^{(-3)}=\chi_{2r}-\chi_{2r-1}=3^{r-1}$$
    and
    $$\ch_{2r-1}^{(-3)}=\chi_{2r-1}-\chi_{2r-2}=3^{r-1}$$
    hence completing the proof.
\end{proof}


\begin{lem}
    For any positive $B\in \R$, there exists $c>B$ and $c<-B$ and an $n\in \N$ such that $\h^n\G{c}=\G{-3}$.
\end{lem}
\begin{proof}
   Observe first that $\G{6}=\G{-3}$. Also note that for $c\ge 6$ the solution $\bar c$ of the equation $h(\bar c) =c$ satisfies $\bar c<-3c$.
   On the other hand, for $ c\le -6$, the largest of the tree solutions $\bar c$ of the equation $h(\bar c) =c$ satisfies $\bar c>-3c-1$.
   Hence by backward iteration and choosing always the solution with the largest absolute value, we can construct an orbit ending in $\G{-3}$ and
   starting at some $\G{c}$ with $c>B$ or  $c<-B$.
   \end{proof}

Similarly as before, we consider the Hesse chains ending in $\Gamma_{\infty}$. 
\begin{lem}
    If $\ch_n^\infty$ denotes the number of Hesse chains of
    length $n$ ending in $\G\infty$, then 
    $$\ch_{2r}^\infty=\ch_{2r-1}^\infty=3^{r-1}$$
    for $r\ge 1$.
\end{lem}
\begin{proof}
Since $\h\G{0}=\G\infty$, the number of Hesse chains ending in $\G\infty$ of length $n$ is the number of zeros of $h^{(n-1)}$.
Therefore the claim follows from Lemma~\ref{number of zeros}.
\end{proof}
%
Now we turn our attention towards Hesse loops.
\begin{prp}
    The only Hesse loop of odd length is the trivial loop $\h\G{-3}=\G{-3}$.
\end{prp}
\begin{proof}
  This follows immediately from Lemma~\ref{components}.
\end{proof}
%
Now, we want to determine the number of Hesse loops of length $n$ for even $n$. We start with the following observation.
\begin{prp}
    For every even $n$, there is at least one Hesse loop of length $n$. 
\end{prp}
\begin{proof}
Recall that $\Phi_n$ denotes the number of fixed points of $h^{(n)}$ (see Proposition~\ref{number of fixed points}).
Now, note that the value of $\Phi_n$ already includes the trivial fixed point $-3$.  So let $\Phi^\prime_n:=\Phi_n-1$ denote
the number of non-trivial fixed points of $h^{(n)}$.
Furthermore, for any $r$, a fixed point of $h^{(r)}$ is also a fixed point of $h^{(mr)}$ for any $m$. 
Also, any loop of length $r$ consists of $r$ elements and contributes this number to $\Phi_n'$. 
It is enough to show that the quantity
    $$\Phi^\prime_{2r}-\sum_{k=1}^{r-1}\Phi^\prime_{2k}$$
    is strictly positive.
    This follows from the fact that for $r>1$, we have
    $$
        \sum_{k=1}^{r-1}\Phi^\prime_{2k}=\sum_{k=1}^{r-1} \left(2\times 3^k - 4\right)
        =3^r-4r+1,
   $$
    and this is indeed strictly smaller than $\Phi_{2r}'=2\times 3^r-4$.
\end{proof}
\noindent{\it Remark.} The above calculation also shows that we must at least have
$$\left\lceil \frac1{2r}\Bigl( \Phi^\prime_{2r}-\sum_{k=1}^{r-1}\Phi^\prime_{2k}\Bigr) \right\rceil=\left \lceil\frac{3^r-5}{2r}\right \rceil+2$$
 loops of length $2r$.

\begin{prp}
    If $\Lambda_n$ denotes the number of Hesse loops of length $n$, then the sequence $\{\Lambda_{2r}\}$ is strictly increasing.
\end{prp}
\begin{proof}
    Since $\Phi^\prime_{2r}$ includes the two elements of the only $2$-loop, the value of $\Lambda_{2r}$ can be at most
    $$\left\lfloor\frac{\Phi^\prime_{2r}-2}{2r}\right\rfloor=\left\lfloor\frac{3^r-3}{r}\right\rfloor$$
    and hence
    $$\left \lceil\frac{3^r-5}{2r}\right \rceil+2\le \Lambda_{2r}\le \left\lfloor\frac{3^r-3}{r}\right\rfloor$$
    for $r>1$.
    So, to prove that the sequence $\{\Lambda_{2r}\}$ is strictly increasing, it is enough to show that
    $$\left \lceil\frac{3^r-5}{2r}\right \rceil>\left\lfloor\frac{3^{r-1}-3}{r-1}\right\rfloor.$$
    This follows from
    $$\frac{3^x-5}{2x}>\frac{3^{x-1}-3}{x-1}$$
    which is true for $x\ge3$ as then we have
    $$3^{x-1}\cdot x+5+x>3^x$$
    which completes the proof for $r\ge3$.

    The cases $r=1,2$ can be checked by hand.
\end{proof}
%


We close this discussion by an explicite formula for the number of loops of length $n=2r$.
\begin{thm}
    The number of loops of length $2r$ is
    $$\Lambda_{2r}=\frac 1{2r} \sum_{d|r} \mu\left(\frac rd\right) \Phi^\prime_{2d}$$
    \textcolor{black}{where  $\Phi_{2d}'=2\times 3^d-4$, and $\mu$ is the M\"obius function.}
\end{thm}
\begin{proof}
    Let the even divisors of $2r$ be $d_1=2,d_2,\dots, d_k=2r$. Since each loop of length $d_m$ contains exactly $d_m$ elements, the total number of fixed points $\neq -3$ of $h^{(2r)}$ is given by
    $$\Phi^\prime_{2r} = \sum_{\substack{d|2r\\ d \text{ even }}} d\cdot \Lambda_d.$$
    The even divisors of $2r$ are twice the divisors of $r$. Hence we may write
    $$\Phi^\prime_{2r}=\sum_{d|r}2d\cdot \Lambda_{2d}.$$
    Using the Möbius inversion formula, we obtain
    \begin{align*}
        &2r\cdot \Lambda_{2r} = \sum_{d|r}\mu\left (\frac rd\right)\Phi^\prime_{2d}
    \end{align*}
    hence completing the proof.
\end{proof}
The sequence $(\Lambda_{2r})$ starts as follows:
$$
\Lambda_{2}=1,\ \Lambda_{4}=3,\ \Lambda_{6}=8,\ \Lambda_{8}=18,\
\Lambda_{10}=48,\ \Lambda_{12}=116,\ \Lambda_{14}=312,\ \Lambda_{16}=810,\ \ldots
$$
The Hesse loop of length 2 is shown in Figure~\ref{fig:galaxy1}.

\section{Hesse derivatives of other normal forms}

So far, we just considered Hesse derivatives of cubic curves in Hesse form,
{\ie}, of curves~$\G{c}$. The reason was that the Hesse derivative of a 
curve in Hesse form is again a curve in Hesse form, which is in general
not the case for curves, for example, in Weierstrass normal form (WNF). 

Below, we first provide curves in WNF
such that their Hesse derivatives are also in~WNF, and then we provide 
cubics in a $D_3$-symmetric form whose Hesse derivatives 
are in the same form, as in Figure~\ref{fig:galaxy1}.

\subsection{Curves in Weierstrass normal form}

Let $\G c:\ x^3+y^3+z^3+c\,xyz=0$ be a cubic curve in Hesse form. 
Then, as described in~\cite[Sec.\,3]{HessePaper}, by a projective 
transformation, the curve $\G c$, where
$$c=-\frac{2 q^3 + 1}{q^2}$$ can be transformed to the curve
$$\G {a,b}: y^2=x^3+a\,x^2+b\,x$$
where $$b=\frac{(q - 1)^3}{q + q^2 + q^3}\qquad\text{and}\qquad
a=\frac{b^2 - 6 b - 3}{4}\,.$$

For example, for $c_0=3 (\sqrt{3}-1)$ we obtain 
$$q_0=-\frac{\sqrt{3}+1}{2},\qquad
b_0=3 + 2 \sqrt{3},\qquad
a_0=0\,,$$
and for $c_1=-\frac{108+c_0^3}{3c_0^2}=-3(\sqrt{3}+1)$ we obtain
$$q_1=\frac{\sqrt{3}-1}{2},\qquad
b_1=3 - 2 \sqrt{3},\qquad
a_1=0\,.$$
Notice that the curves $\G {a_0,b_0}$ and $\G {a_1,b_1}$
form a Hesse loop of length~$2$.

\subsection{Curves in $D_3$-symmetric form}

In \cite[Sec.\,2]{hhh}, a $D_3$-symmetric form of cubic
curves was introduced and it was shown how to transform a curve
in WNF with a projective transformation into a curve in 
$D_3$-symmetric form. Now, with a similar projective transformation
we can transform any curve $\G c:\ x^3+y^3+z^3+c\,xyz=0$ in Hesse form
directly into the curve 
$$x^3 - 3\,x y^2 + \frac{\sqrt{27} (c - 6)}{2 (c + 3)}\,(x^2 + y^2) - 
\frac{\sqrt{27}}{2}=0\,,$$ where the latter curve is $D_3$-symmetric
(like the curves in Figure\;\ref{fig:galaxy1}). 

\bibliographystyle{plain}

\end{document}